\documentclass[12pt, letterpaper, twoside]{article}
\usepackage[utf8]{inputenc}

\usepackage{comment}
\usepackage{amssymb}
\usepackage{amsmath}
\usepackage{amsthm}
\usepackage{xcolor}
\newtheorem{lemma}{Lemma}
\newtheorem{corollary}{Corollary}
\newtheorem{example}{Example}
\newtheorem{proposition}{Proposition}

\usepackage[ruled,vlined]{algorithm2e}

\newcommand{\PAF}{\operatorname{PAF}}
\newcommand{\PSD}{\operatorname{PSD}}
\newcommand{\DFT}{\operatorname{DFT}}
\renewcommand{\mod}{\mathrm{mod}\;}
\newcommand{\pam}{\widehat{A}_m}
\newcommand{\pbm}{\widehat{B}_m}
\newcommand{\pa}[1]{\widehat{A}_{#1}}
\newcommand{\pb}[1]{\widehat{B}_{#1}}
\newcommand{\Z}{\mathbb{Z}}

\title{Legendre pairs of lengths $\ell \equiv 0 \, (\mod 3)$}
\author{Ilias Kotsireas\thanks{Corresponding author. Supported by an NSERC grant.}  $\,\,$ \& $\,\,$
Christoph Koutschan\thanks{Supported by the Austrian Science Fund (FWF): F5011-N15.\newline
This paper has been published in Journal of Combinatorial Designs 29:870--887,
DOI: 10.1002/jcd.21806}}
\date{\today}

\begin{document}

\maketitle

\begin{abstract}
We prove a proposition that connects constant-PAF sequences and the corresponding Legendre pairs with integer PSD values.
We show how to determine explicitly the complete spectrum of the $({\ell}/{3})$-rd value of the discrete Fourier transform for Legendre pairs of lengths $\ell \equiv 0 \, (\mod 3)$. This is accomplished by two new algorithms based on number-theoretic arguments. As an application, we prove that Legendre pairs of the open lengths $117$, $129$, $133$, and~$147$ exist by finding Legendre pairs of these lengths with a multiplier group of order at least~$3$. As a consequence, $85$, $87$, $115$, $145$, $159$, $161$, $169$, $175$, $177$, $185$, $187$, $195$ are the twelve integers in the range $< 200$ for which the question of existence of Legendre pairs remains unsolved. 
\end{abstract}


\section{Introduction}

Let $A$ denote a finite sequence $A = [a_1,\ldots,a_{\ell}]$ of length $\ell$. 

\noindent The periodic autocorrelation function (PAF) of $A$ at lag $s$ is defined as
\begin{equation}
    \PAF(A,s) = \displaystyle\sum_{i=1}^{\ell} a_i\, a_{i+s}, \quad \forall \,\, s = 0, \ldots, \ell-1,
    \label{definition:PAF}
\end{equation}
where $i+s$ is taken modulo $\ell$, when $i+s>\ell$. 

\noindent The discrete Fourier transform (DFT) of $A$ at lag $s$ is defined as
\begin{equation}
    \DFT(A,s) = \displaystyle\sum_{i=1}^{\ell} a_i\, \omega^{s \cdot (i-1)}, \quad \forall \,\, s = 1, \ldots, \ell,
    \label{definition:DFT}
\end{equation}
where $\omega = \cos\left({2\pi}/{\ell}\right) + i \sin\left({2\pi}/{\ell}\right)$ is the primitive $\ell$-th root of unity, that satisfies $\omega^{\ell} =1$. 

\noindent The power spectral density (PSD) of $A$ at lag $s$ is defined as 
\begin{equation}
    \PSD(A,s) = \bigl|\DFT(A,s)\bigr|^2 = \Re(\DFT(A,s))^2 + \Im(\DFT(A,s))^2, \quad \forall \,\, s = 1, \ldots, \ell
    \label{definition:PSD}
\end{equation}
i.e., the PSD values are defined as the sum of squares of the real and imaginary parts of the DFT values.

Let $\ell$ be an odd positive integer. Two sequences $A = [a_1,\ldots,a_\ell]$ and $B = [b_1,\ldots,b_\ell]$ of length~$\ell$ and consisting of elements from $\{-1,+1\}$, such that
$a_1 + \ldots + a_\ell = b_1 + \ldots + b_\ell = \pm 1$
form a Legendre pair of length $\ell$ if 
\begin{equation}
    \PAF(A,s) + \PAF(B,s) = -2, \quad \forall \,\, s = 1,\ldots,\frac{\ell-1}{2} .
    \label{PAF_minus_2}
\end{equation}
In the context of Legendre pairs, we typically work with the sole assumption that $a_1 + \ldots + a_\ell = 1$ and $b_1 + \ldots + b_\ell = 1$, without loss of generality. It is well-known, see \cite{FGS:2001}, that if $(A,B)$ form a Legendre pair of length $\ell$, then we have
\begin{equation}
    \PSD(A,s) + \PSD(B,s) = 2\ell + 2, \quad \forall \,\, s = 1,\ldots,\frac{\ell-1}{2} . 
    \label{PSD_2ell_plus2}
\end{equation}

The paper~\cite{FGS:2001} is fundamental in the study of Legendre pairs, 
as it initiated the use of the PSD criterion, in the search for Legendre pairs. 
More specifically, the PSD criterion asserts that if, in the course of a search algorithm, an index $i$ in the range $1,\ldots,(\ell-1)/{2}$ is detected, such that $\PSD(A,i) > 2\ell + 2$, then the corresponding (candidate) sequence~$A$ can be discarded from the search, because it is unsuitable to form a Legendre pair. This is due to the fact that the PSD values are always non-negative, as sums of norm squares. Given a Legendre pair of length $\ell$, one can construct a Hadamard matrix of order $2\ell + 2$, using a two circulant core template array found in~\cite{FGS:2001}. 

Throughout this paper, we use the notation $\Z_{\ell}^{\star}$ to denote the multiplicative group $\{j\in\Z_{\ell}\mid\gcd(j,\ell)=1\}$. Let $I\subseteq\Z_{\ell}$, then an element $t\in\Z_{\ell}^{\star}$ is called a multiplier of~$I$ if there exists $g\in\Z_{\ell}$ such that
\[
  t\cdot I = I + g,
\]
where $I+g:=\{i+g\mid i\in I\}$ and analogously for $t\cdot I$. We say that $t$ is a multiplier for a sequence $A=[a_1,\dots,a_{\ell}]\in\{-1,+1\}^{\ell}$ if it is a multiplier of $I:=\{i\in\Z_{\ell}\mid a_i=1\}$. See \cite{JT:2021_JACO} for more details. In this paper, we restrict our searches for Legendre pairs to sequences whose group of multipliers contains a prespecified subgroup of $\Z_{\ell}^{\star}$, also known as the union-of-orbits approach. In most instances considered here, we specify a subgroup of size~$3$, so that the search space is neither too restrictive causing no Legendre pairs to be found, nor too large causing the algorithm to get stuck in parts of the search space that contain no Legendre pairs given the current computational resources.

The rest of the paper is organized as follows. In Section~\ref{sec:0mod3}, we present some theoretical results on the possible PSD values of sequences in Legendre pairs, under the assumption that their length~$\ell$ is divisible by~$3$. These results are then exploited in Section~\ref{sec:comp}, where we use them as additional filter criteria in order to speed up our exhaustive searches for Legendre pairs. With the help of considerable computational resources, we succeeded to find Legendre pairs of lengths $\ell=117$, $\ell=129$, and $\ell=147$. It was unknown until now whether Legendre pairs of these lengths existed or not (see Sections \ref{sec:117} -- \ref{sec:147}). As an encore, in Section~\ref{sec:133} we hint at the possibility of extending our ideas to lengths~$\ell$ that are not necessarily divisible by~$3$ but by some other small prime, and for the first time present some examples of Legendre pairs of length $\ell=133$.

\section{Legendre pairs of length $\ell \equiv 0 \, (\mod 3)$}
\label{sec:0mod3}

Consider $(A,B)$ a Legendre pair of length $\ell$ such that $\ell \equiv 0 \, (\mod 3)$ and set $m={\ell}/{3}$. The following lemma is proved in \cite{AnnComb_2009_KKS}

\begin{lemma}
Let $\ell$ be an odd integer such that $\ell \equiv 0 \, (\mod 3)$
and let $m = {\ell}/{3}$. Let $A = [a_1,\ldots,a_\ell]$ be a $\{ -1, +1 \}$-sequence. Then 
$$
    \DFT(A,m) =
    \left(A_1 -\frac{1}{2} A_2 -\frac{1}{2} A_3\right)
    +
    \left(\frac{\sqrt{3}}{2} A_2 - \frac{\sqrt{3}}{2} A_3\right) i,
$$
$$
    \PSD(A,m) =
    A_1^2 + A_2^2 + A_3^2 - A_1 A_2 - A_1 A_3 - A_2 A_3,
$$
where
$$
    A_1 = \sum_{i=0}^{m-1} a_{3i+1}, \,\,\,
    A_2 = \sum_{i=0}^{m-1} a_{3i+2}, \,\,\,
    A_3 = \sum_{i=0}^{m-1} a_{3i+3}.
$$
\label{Lemma:PSD_length_divisible_by_three}
\end{lemma}
The proof of Lemma~\ref{Lemma:PSD_length_divisible_by_three} is based on the exact evaluation of the roots of the cyclotomic polynomial of degree $3$.

\noindent 
Let $\mathcal{A} = \{ a_1, \ldots, a_{\ell} \}$, and 
\[
  e_2(\mathcal{A}) = \sum_{i < j} a_i a_j
\]
denote the second elementary symmetric function on~$\mathcal{A}$. Let
\[
  p_1(\mathcal{A}) = \sum_{i=1}^{\ell} a_i
  \quad\text{and}\quad
  p_2(\mathcal{A}) = \displaystyle\sum_{i=1}^{\ell} a_i^2
\]
denote the first and second power sums on~$\mathcal{A}$.
The following special case of the Jacobi-Trudi identity
\begin{equation}
                    e_2(\mathcal{A}) = \frac{p_1(\mathcal{A})^2}{2} -\frac{p_2(\mathcal{A})}{2}
\label{Jacobi-Trudi}
\end{equation}
can be found in~\cite{Strurmfels:2008}.

Applying Lemma~\ref{Lemma:PSD_length_divisible_by_three} to 
a Legendre pair $(A,B)$ of length $\ell$ such that $\ell \equiv 0 \, (\mod 3)$, we obtain the following:

\begin{corollary}
If $\ell \equiv 0 \, (\mod 3)$, $m={\ell}/{3}$, and if the two $\{-1,+1\}$-sequences $A=[a_1,\ldots,a_\ell]$ and
$B=[b_1,\ldots,b_\ell]$ form a Legendre pair of length $\ell$, then
\begin{equation}
\left\{
\begin{array}{l}
\PSD(A,m) = \displaystyle\frac{3}{2} \left( A_1^2 +
A_2^2 + A_3^2 \right) -\displaystyle\frac{1}{2} \\
  \\
\PSD(B,m) = \displaystyle\frac{3}{2} \left( B_1^2 +
B_2^2 + B_3^2 \right) -\displaystyle\frac{1}{2}
\end{array}
\right. \label{PSD-values-pam-pbm}
\end{equation}
\begin{equation}
    A_1^2 + A_2^2 + A_3^2 + B_1^2 + B_2^2 + B_3^2 = 4m+2
    \label{Six_Squares_Diophantine_Equation_eq_4mp2}
\end{equation}
\label{Corollary-PSD-values}
\end{corollary}
where
\begin{alignat*}{3}
  A_1 &= \sum_{i=0}^{m-1} a_{3i+1}, \quad&
  A_2 &= \sum_{i=0}^{m-1} a_{3i+2}, \quad&
  A_3 &= \sum_{i=0}^{m-1} a_{3i+3},
  \\[1ex]
  B_1 &= \sum_{i=0}^{m-1} b_{3i+1}, &
  B_2 &= \sum_{i=0}^{m-1} b_{3i+2}, &
  B_3 &= \sum_{i=0}^{m-1} b_{3i+3}.
\end{alignat*}
\begin{proof}
Applying Lemma~\ref{Lemma:PSD_length_divisible_by_three} to the sequences $A$, $B$ separately we obtain
\begin{align*} 
    \PSD(A,m) &=
    A_1^2 + A_2^2 + A_3^2 - \underbrace{(A_1 A_2 + A_1 A_3 + A_2 A_3)}_{e_2(A_1,A_2,A_3)},
  \\
    \PSD(B,m) &=
    B_1^2 + B_2^2 + B_3^2 - \underbrace{(B_1 B_2 + B_1 B_3 + B_2 B_3)}_{e_2(B_1,B_2,B_3)}.
\end{align*}
The second elementary symmetric functions $e_2(A_1,A_2,A_3)$ and $e_2(B_1,B_2,B_3)$ are related to the first elementary symmetric functions $e_1(A_1,A_2,A_3)$ and $e_1(B_1,B_2,B_3)$ via the special case of the Jacobi-Trudi identity (\ref{Jacobi-Trudi}). We also know that 
$e_1(A_1,A_2,A_3) = A_1 + A_2 + A_3 = 1$ and $e_1(B_1,B_2,B_3) = B_1 + B_2 + B_3 = 1$. Therefore we obtain (\ref{PSD-values-pam-pbm}), and
\begin{align*}
  A_1^2 + A_2^2 + A_3^2 + B_1^2 + B_2^2 + B_3^2 &=
  \frac{2\PSD(A,m)+1}{3} + \frac{2\PSD(B,m)+1}{3} \\[1ex]
  &= \frac{2(2\ell+2)+2}{3} = \frac{4\ell+6}{3} = 4m+2.\qedhere
\end{align*}
\end{proof}

In the sequel, we denote $\PSD(A,m)$ by $\pam$ and $\PSD(B,m)$ by $\pbm$. 

Corollary~\ref{Corollary-PSD-values} can be used to derive additional decoupled constraints (i.e., involving $A_i$ and $B_i$ separately) based on (\ref{Six_Squares_Diophantine_Equation_eq_4mp2}). From~(\ref{PSD_2ell_plus2})
we know:
\begin{equation}
    \pam + \pbm = 2\ell+2.
    \label{Equation: pam-pbm-eq-2n-2}
\end{equation}
Moreover, from (\ref{PSD-values-pam-pbm}) we obtain:
\begin{equation}
    A_1^2 + A_2^2 + A_3^2 = \frac{2\pam+1}{3}
    \quad\text{and}\quad
    B_1^2 + B_2^2 + B_3^2 = \frac{2\pbm+1}{3}.
\label{Equations:3squares}
\end{equation}
Since both these sums of three odd integer squares are integers,
we obtain that $2\pam + 1 \equiv 9\, (\mod 24)$ and
$2\pbm + 1 \equiv 9\, (\mod 24)$
i.e.
$\pam \equiv 4\, (\mod 12)$ and 
$\pbm \equiv 4\, (\mod 12)$.
Therefore, the set of possible pairs of values $\bigl[\pam,\pbm\bigr]$ can be restricted considerably. In addition, a possible pair of values $\bigl[\pam,\pbm\bigr]$ has to
be compatible with the linear constraints
\begin{equation}
    A_1 + A_2 + A_3 = 1 \quad\text{and}\quad B_1 + B_2 + B_3 = 1.
\label{LinearConstraints}
\end{equation}
For a Legendre pair $(A,B)$ of length $\ell$, we must have that $A_1,A_2,A_3,B_1,B_2,B_3$ are all odd. For given fixed values of $\pam$, $\pbm$, equations (\ref{Equations:3squares}) can be solved independently as sums-of-squares Diophantine equations and typically have anywhere from $1$ to $5$ all-odd solutions (up to sign), for the right-hand-side values that arise in the context of Legendre pairs of lengths $\ell < 200$. 
These solutions give possible triplets of values for $(A_1,A_2,A_3)$ and $(B_1,B_2,B_3)$ that must be compatible with the linear constraints
(\ref{LinearConstraints}). 
The above discussion suffices to formulate an algorithm for determining explicitly the complete spectrum of the $({\ell}/{3})$-rd PSD values for any Legendre pair of length $\ell$ divisible by three. We outline this algorithm below.
\bigskip

\begin{algorithm}[H]
\SetAlgoLined
Input: An odd positive integer $\ell = 3 \cdot m$ \;  
Initialization: ${\cal S} = \{\}$  \; 
 \For{$s=0,\ldots,\lfloor(\ell-1)/6\rfloor$}{
   (1) form the candidate $\bigl[\pam,\pbm\bigr]$ pair $[12s + 4, 2\ell + 2 - (12s + 4) ]$ \;
   (2) compute the values of $A_1^2 + A_2^2 + A_3^2$ and $B_1^2 + B_2^2 + B_3^2$ using (\ref{Equations:3squares}) \;
   (3) solve (up to sign) the two 
   sum-of-squares Diophantine equations\\[1ex]
    $\qquad A_1^2 + A_2^2 + A_3^2 = \dfrac{2 (12s + 4) + 1}{3}$, \\[1ex]
    $\qquad B_1^2 + B_2^2 + B_3^2 = \dfrac{2(2\ell + 2 - (12s + 4)) + 1}{3}$ \;
   \vspace{1ex}
  \eIf{there are all-odd solutions of these two Diophantine equations, compatible with the linear constraints (\ref{LinearConstraints})  }{
   insert the pair $[12s + 4, 2\ell + 2 - (12s + 4)]$ in ${\cal S}$, as an element of the spectrum of $\bigl[\pam,\pbm\bigr]$ \;
   }{
   discard the pair $[12s + 4, 2\ell + 2 - (12s + 4)]$ as it cannot be an element of the spectrum of $\bigl[\pam,\pbm\bigr]$ \;
  }
 }
 Output: the spectrum ${ \cal S}$ of pairs of values $\bigl[\pam,\pbm\bigr]$ for Legendre pairs $(A,B)$ of length $\ell = 3 \cdot m$ \;  
 \caption{Determination of the spectrum ${\cal S}$ }
 \label{algo.pam_pbm_spectrum}
\end{algorithm}

\begin{example}\label{ex.pam_pbm_117}
We illustrate Algorithm~\ref{algo.pam_pbm_spectrum} with the case $\ell = 117 = 3 \cdot 39$, i.e., $m=39$.
First, we have $\pam + \pbm = 2 \cdot 117 + 2 = 236$ and in addition $\pam \equiv 4 \,\,(\mod 12)$ and $\pbm \equiv 4 \,\,(\mod 12)$. Given that every pair of values $\bigl[\pam,\pbm\bigr]$ determines the values of $A_1^2 + A_2^2 + A_3^2$ and $B_1^2 + B_2^2 + B_3^2$ via (\ref{Equations:3squares}), we obtain Table~\ref{tab:spectrum}.
The first row of Table~\ref{tab:spectrum} indicates a reason why a certain 
$\bigl[\pam,\pbm\bigr]$ combination can be discarded. 
The last three rows of Table~\ref{tab:spectrum} indicate the only three $\bigl[\pam,\pbm\bigr]$ combinations that can possibly hold. 
The remaining rows of the table corresponding to all other $\bigl[\pam,\pbm\bigr]$ combinations are omitted. The omitted rows do not lead to compatible assignments for $A_1,A_2,A_3$ and/or $B_1,B_2,B_3$. Only $3$ pairs of values are not ruled out to occur in Legendre pairs of length $117$. This allows us to add an additional layer of parallelism when searching for such Legendre pairs. 
\end{example}

\begin{table}
$$
\begin{array}{l|l@{}}
\bigl[\pam,\pbm\bigr] & \\
\hline\rule{0pt}{12pt}
(4,232) & A_1^2 + A_2^2 + A_3^2 = 3, \leadsto [1, 1, 1] \\
        & B_1^2 + B_2^2 + B_3^2 = 155, \leadsto [3, 5, 11], [5, 7, 9] \leadsto \mbox{no compatible assignments} \\ 
\hline\rule{0pt}{12pt}
(28,208)  & A_1^2 + A_2^2 + A_3^2 = 19, \leadsto [1, 3, 3] \\
        & B_1^2 + B_2^2 + B_3^2 = 139, \leadsto [3, 3, 11], [3, 7, 9] \\
        & \mbox{compatible assignments: }
        (A_1,A_2,A_3) = (1,-3,3), (B_1,B_2,B_3) = (3,7,-9) \\
\hline\rule{0pt}{12pt}
(64, 172)  & A_1^2 + A_2^2 + A_3^2 = 43, \leadsto [3, 3, 5] \\
           & B_1^2 + B_2^2 + B_3^2 = 115, \leadsto [3, 5, 9] \\
           & \mbox{compatible assignments: }
           (A_1,A_2,A_3) = (3,3,-5), (B_1,B_2,B_3) = (-3,-5,9) \\
\hline\rule{0pt}{12pt}
(112, 124)  & A_1^2 + A_2^2 + A_3^2 = 75, \leadsto [1, 5, 7], [5, 5, 5] \\
            & B_1^2 + B_2^2 + B_3^2 = 83, \leadsto [1, 1, 9], [3, 5, 7] \\
            & \mbox{compatible assignments: }
           (A_1,A_2,A_3) = (-1,-5,7), (B_1,B_2,B_3) = (3,5,-7) \\
\hline
\end{array}
$$
\caption{Some computations using Algorithm~\ref{algo.pam_pbm_spectrum} for the spectrum of $\bigl[\pam,\pbm\bigr]$ for $m=39$.}
\label{tab:spectrum}
\end{table}

If $H$ is a subgroup of $\Z_{\ell}^{\star}$ of size~$3$, then $H$ is cyclic and all of its
elements must be $1\,(\mod 3)$. Moreover, if $H$ is a subgroup of $\Z_{\ell}^{\star}$ with
all its members equal to $1\,(\mod 3)$, then each orbit of~$H$ consists of elements that are equal to each other $(\mod 3)$. The consequence is that each orbit contributes to exactly one of the three quantities $A_1,A_2,A_3$ that were defined in Lemma~\ref{Lemma:PSD_length_divisible_by_three}. We can exploit this observation to further confine the potential values for $\pam$ and $\pbm$. The following algorithm is formulated under the assumption that the chosen orbits indicate the positions of the $+1$'s, but since $\pam$ is a sum-of-squares in the $A_i$, it works also in situations where the chosen orbits mark the positions of the $-1$'s.

\begin{algorithm}[H]
\SetAlgoLined
Input: An odd positive integer $\ell = 3 \cdot m$, a subgroup $H$ of $\Z_{\ell}^{\star}$ s.t. $h\equiv 1\,(\mod 3)$ for all $h\in H$, and non-negative integers $c_1,\dots,c_t$ indicating the number of chosen orbits of sizes $s_1,\dots,s_t$, respectively \;
\For{$i=1,\dots,t$}{\For{$j\in\{0,1,2\}$}{
  $n_{i,j} = $ number of orbits of size $s_i$ with elements $\equiv j\,(\mod 3)$ \;
}}
$T = \bigl\{ (k_{1,1},\dots,k_{t,1},k_{1,2},\dots,k_{t,2}) \mathrel{\big|}$ \\ 
  \qquad\quad $0\leq k_{1,1}\leq\min\{c_1,n_{1,1}\},\dots,
    0\leq k_{t,1}\leq\min\{c_t,n_{t,1}\},$ \\
  \qquad\quad $0\leq k_{1,2}\leq\min\{c_1,n_{1,2}\},\dots,
    0\leq k_{t,2}\leq\min\{c_t,n_{t,2}\},$ \\
  \qquad\quad $0\leq c_1-k_{1,1}-k_{1,2}\leq n_{1,0},\dots,
    0\leq c_t-k_{t,1}-k_{t,2}\leq n_{t,0}
\bigr\}$ \;
${\cal C} = \{\}$  \;
\For{$(k_{1,1},\dots,k_{t,1},k_{1,2},\dots,k_{t,2})\in T$}{
    $A_1 = -m + 2\cdot\sum_{i=1}^t s_i\cdot k_{i,1}$ \;
    $A_2 = -m + 2\cdot\sum_{i=1}^t s_i\cdot k_{i,2}$ \;
    $A_3 = -m + 2\cdot\sum_{i=1}^t s_i\cdot(c_i-k_{i,1}-k_{i,2})$ \;
    $P = A_1^2 + A_2^2 + A_3^2 - A_1 A_2 - A_1 A_3 - A_2 A_3$ \;
    ${\cal C} = {\cal C} \cup \{P\}$ \;
}
Output: the set ${ \cal C}$ of potential values for $\pam$ that are compatible with the choice of $c_1,\dots,c_t$ orbits of~$H$ \;  
\caption{Determination of PSD values $\pam$ compatible with the $H$ orbits}
\label{algo.pam}
\end{algorithm}

\begin{example}\label{ex.pam_117_H1}
Continuing Example~\ref{ex.pam_pbm_117} for $\ell=117$, we apply Algorithm~\ref{algo.pam} in order to show that $\bigl[\pa{39},\pb{39}\bigr]=[112,124]$ cannot appear when we employ the subgroup $H_1=\{1,16,22\}$ for conducting a search with the orbits method. The subgroup $H_1$ induces $2$ orbits of size $s_1=1$, and $38$ orbits of size $s_2=3$. For such a search one may choose $c_1=2$ orbits of size~$1$ and $c_2=19$ orbits of size~$3$. By looking at the orbits (they are listed explicitly below in Section~\ref{sec:117H1}), we find $n_{2,0}=12$ orbits whose elements are divisible by~$3$, and similarly $n_{2,1}=n_{2,2}=13$. Moreover, $n_{1,0}=2$ and $n_{1,1}=n_{1,2}=0$, which eventually implies $k_{1,1}=k_{1,2}=0$. Let $k_{2,1}$ (resp.~$k_{2,2}$) denote the number of chosen $3$-orbits whose elements are $1\,(\mod 3)$ (resp.~$2\,(\mod 3)$). Then we obtain
\[
  A_1 = 6\cdot k_{2,1} - 39,\quad
  A_2 = 6\cdot k_{2,2} - 39,\quad
  A_3 = 2\cdot2 + 6\cdot(19 - k_{2,1} - k_{2,2}) - 39.
\]
Letting $k_{2,1}$ and $k_{2,2}$ range over all admissible values, i.e.,
\[
  0\leq k_{2,1}\leq 13 \;\land\; 0\leq k_{2,2}\leq 13 \;\land\;
  0\leq 19-k_{2,1}-k_{2,2}\leq 12,
\]
we get all possible values for $A_1,A_2,A_3$ and determine the potential values
\[
  \pa{39} = \PSD(A,39) = A_1^2 + A_2^2 + A_3^2 - A_1 A_2 - A_1 A_3 - A_2 A_3.
\]
to be
\[
  28,\, 64,\, 100,\, 172,\, 208,\, 244,\, 316,\, 388,\, 496, \dots,
  4132,\, 4348,\, 4564.
\]
This list excludes the possibility of finding a Legendre pair with $\bigl[\pa{39},\pb{39}\bigr]=[112,124]$. Note that this result does not change if we set $c_1=1$ and $c_2=19$ (a situation where the chosen orbits indicate the positions of the $-1$'s).
\end{example}

\section{Compression and constant-PAF sequences}
The term ``constant-PAF sequences'' is taken to mean sequences all of whose PAF values are equal to the same constant. We refer the reader to \cite{DK:DCC:2015} for the definition and properties of compression of Legendre pairs. 
It has been observed experimentally in the current paper, as well as in \cite{JT:2021}, that some Legendre pairs of composite length $\ell = n \cdot m$ have the properties that: 
\begin{itemize}
    \item their $m$-compression is made up from two constant-PAF sequences of length $n$ 
    \item some of the PSD values of the resulting Legendre pairs of length $\ell$ are integers.
\end{itemize}
In this section, we prove a proposition that elucidates the connection between these two aforementioned facts. 

\begin{proposition}
Let $A=[a_1,\ldots,a_{\ell}]$ and $B=[b_1,\ldots,b_{\ell}]$ be a Legendre pair 
of composite length $\ell = n \cdot m$. Let $\mathcal{A} = [A_1,\ldots,A_n]$, 
$\mathcal{B} = [B_1,\ldots,B_n]$, where
\[
  A_j = \sum_{i=0}^{m-1} a_{ni+j}
  \quad\text{and}\quad
  B_j = \sum_{i=0}^{m-1} b_{ni+j}
\]
for $j=1,\ldots,n$, i.e. $(\mathcal{A},\mathcal{B})$ is the $m$-compression of $(A,B)$. \\ 
If the $m$-compression of $(A,B)$ is made up from two constant-PAF sequences of length $n$: 
$$  
    \PAF(\mathcal{A},1) = \PAF(\mathcal{A},2) = \cdots = \PAF(\mathcal{A},\frac{n-1}{2}) 
$$
$$
    \PAF(\mathcal{B},1) = \PAF(\mathcal{B},2) = \cdots = \PAF(\mathcal{B},\frac{n-1}{2})
$$
(where $ \PAF(\mathcal{A},1) + \PAF(\mathcal{B},1) = (-2) \cdot m$),
then the PSD values at integer multiples of $m$ of $A$ and $B$ are integers, with the explicit evaluations
$$
\PSD(A,m \cdot s) = p_2(\mathcal{A}) - \PAF(\mathcal{A},1), \quad s = 1,2,\ldots,\frac{n-1}{2}
$$
$$
\PSD(B,m \cdot s) = p_2(\mathcal{B}) - \PAF(\mathcal{B},1), \quad s = 1,2,\ldots,\frac{n-1}{2}
$$
(where $\PSD(A,m \cdot s) + \PSD(B,m \cdot s) = 2\cdot \ell +2$). 
\label{Proposition:integerPSD}
\end{proposition}
\begin{proof}
We use the fact that the PSD remains invariant under $m$-compression, see \cite{DK:DCC:2015}. We also use the Wiener-Khinchin theorem, see \cite{FGS:2001}, that states that the PSD of a sequence is equal to the DFT of its periodic autocorrelation function. We also use the fact that certain sums of roots of unity vanish identically. For every $s = 1,2,\ldots,(n-1)/2$ and $\omega$ the primitive $n$-th root of unity we have:
$$
\begin{array}{lcl}
\PSD(A,m \cdot s) & = & \PSD(\mathcal{A},s)  \\
                  & = & 
\displaystyle\sum_{j=0}^{n-1} \PAF(\mathcal{A},j) \, \omega^{js}   \\
                  & = & 
\PAF(\mathcal{A},0) + \PAF(\mathcal{A},1) \left( \displaystyle\sum_{j=1}^{n-1} \omega^{js} \right) \\
                  & = & p_2(\mathcal{A}) -\PAF(\mathcal{A},1).
\end{array}
$$
The assertion $\PSD(B,m \cdot s) = p_2(\mathcal{B}) - \PAF(\mathcal{B},1)$ is proved in a completely analogous manner.
\end{proof}

We remark that the roles of $n$ and $m$ in Proposition~\ref{Proposition:integerPSD} are not interchangeable. Proposition~\ref{Proposition:integerPSD} will be illustrated in the next section. 

\section{Computational results}
\label{sec:comp}


We have implemented the systematic traversal of the search space in the {\tt C} language, gaining (not surprisingly) a considerable speed-up compared to our prototype implementations in Maple and Mathematica. For each sequence~$A$ in the search space, we first apply Lemma~\ref{Lemma:PSD_length_divisible_by_three} (provided that $\ell\equiv0\,\,(\mod 3)$), by computing the sums $A_1,A_2,A_3$ and then $\PSD(A,{\ell}/{3})$ in exact arithmetic. If a sequence passes this test (or if $\ell\not\equiv0\,\,(\mod 3)$), our program continues with the full PSD test, i.e., it checks whether $\PSD(A,k)\leq2\ell+2$ for all $1\leq k\leq(\ell-1)/2$ (note that we can exploit early termination here). The DFT is computed in floating point arithmetic using double precision. For each sequence passing this second test, we write the two sequences $\bigl(\PSD(A,k)\bigr)_{k\in I}$ and $\bigl(2\ell+2-\PSD(A,k)\bigr)_{k\in I}$ with $I=\{1,\dots,(\ell-1)/2\}\setminus\{{\ell}/{3}\}$ into an output file. Since the PSD values are floating point numbers, we convert them to integers and, in order to save disk space, hash them modulo 16. The results are then saved as hexadecimal strings of length~$|I|$. A Legendre pair corresponds to two lines in the output file whose two strings match pairwise (but in reverse order). Due to the hashing there is the possibility to find matches which do not correspond to Legendre pairs, but the probability that this happens is negligible and such false candidates can easily be sorted out in a post-processing step.

All times were measured on RICAM's computing cluster \texttt{radon1}, which has 1168 Xeon E5-2630v3 (2.4Ghz) threads. For the reported computations, we employed a moderate parallelization, typically using 16 threads for one task. Since the parallelization is done by splitting the search space into pieces, it scales very well. The reported times are given in CPU hours, i.e., as the sum of the times of each thread.

\subsection{Legendre pairs of length $117$}
\label{sec:117}

We executed Algorithm~\ref{algo.pam_pbm_spectrum} for Legendre pairs of length $\ell = 117 = 3 \cdot 39$ and obtained 
$$
  \bigl[\PSD(A,39), \PSD(B,39)\bigr] \in \bigl\{ [28, 208], [64, 172], [112, 124] \bigr\},
$$
as in Example~\ref{ex.pam_pbm_117}.
There are four subgroups of order 3 in $\Z_{117}^\star$
$$
   H_1 = \{1, 16, 22\},\quad  H_2 = \{1, 40, 79\},\quad
   H_3 = \{1, 55, 100\},\quad H_4 = \{1, 61, 94\} .
$$
In the following subsections we investigate these subgroups separately, by considering only
sequences whose multiplier group contains the respective subgroup.

\subsubsection{Legendre pairs of length $117$ via $H_1$}\label{sec:117H1}
The subgroup $H_1 = \{1, 16, 22\}$ of order 3 of $\Z_{117}^\star$ acts on $\Z_{117}$ and yields 
$38$ orbits of size $3$ and $2$ orbits of size $1$.
We list the $38+2$ orbits of the action of $H_1 = \{1, 16, 22\}$ on $\Z_{117}$ as follows:
$$
\begin{array}{l@{\qquad}l@{\qquad}l}
H_1 \cdot 1 = \{1, 16, 22\}, & H_1 \cdot 2 = \{2, 32, 44\}, & H_1 \cdot 3 = \{3, 48, 66\}, \\
H_1 \cdot 4 =\{4, 64, 88\}, & H_1 \cdot 5 = \{5, 80, 110\}, & H_1 \cdot 6 = \{6, 15, 96\}, \\
H_1 \cdot 7 = \{7, 37, 112\}, & H_1 \cdot 8 = \{8, 11, 59\}, & H_1 \cdot 9 = \{9, 27, 81\}, \\
H_1 \cdot 10 = \{10, 43, 103\}, & H_1 \cdot 12 = \{12, 30, 75\}, & H_1 \cdot 13 =  \{13, 52, 91\}, \\
H_1 \cdot 14 = \{14, 74, 107\}, & H_1 \cdot 17 = \{17, 23, 38\}, & H_1 \cdot 18 = \{18, 45, 54\}, \\
H_1 \cdot 19 = \{19, 67, 70\}, & H_1 \cdot 20 = \{20, 86, 89\}, & H_1 \cdot 21 = \{21, 102, 111\}, \\
H_1 \cdot 24 = \{24, 33, 60\}, & H_1 \cdot 25 = \{25, 49, 82\}, & H_1 \cdot 26 = \{26, 65, 104\}, \\
H_1 \cdot 28 = \{28, 31, 97\}, & H_1 \cdot 29 = \{29, 53, 113\}, & H_1 \cdot 34 = \{34, 46, 76\}, \\
H_1 \cdot 35 = \{35, 68, 92\}, & H_1 \cdot 36 = \{36, 90, 108\}, & H_1 \cdot 40 = \{40, 55, 61\}, \\
H_1 \cdot 41 = \{41, 71, 83\}, & H_1 \cdot 42 = \{42, 87, 105\}, & H_1 \cdot 47 = \{47, 50, 98\}, \\
H_1 \cdot 51 = \{51, 69, 114\}, & H_1 \cdot 56 = \{56, 62, 77\}, & H_1 \cdot 57 = \{57, 84, 93\}, \\
H_1 \cdot 58 = \{58, 106, 109\}, & H_1 \cdot 63 = \{63, 72, 99\}, & H_1 \cdot 73 = \{73, 85, 115\}, \\
H_1 \cdot 79 = \{79, 94, 100\}, & H_1 \cdot 95 = \{95, 101, 116\}, & \\
H_1 \cdot 39 = \{39\}, & H_1 \cdot 78 = \{78\} . &
\end{array}
$$
Subsequently, we distinguish two cases:
\begin{itemize}
    \item Case (I): make use of $2$ orbits of size $1$ and $19$ orbits of size $3$, to make a subset of size $2 \cdot 1 + 19 \cdot 3 = 59 = (117+1)/{2}$. The search space is of size: $\binom{2}{2} \cdot \binom{38}{19} = 35{,}345{,}263{,}800$.
    \item Case (II): make use of $1$ orbit of size $1$ and $19$ orbits of size $3$, to make a subset of size $1 \cdot 1 + 19 \cdot 3 = 58 = (117-1)/{2}$. The search space is of size: $\binom{2}{1} \cdot \binom{38}{19} = 70{,}690{,}527{,}600$. 
    In order to have sequences whose entries sum up to~$1$, these orbits have to encode the positions of the $-1$'s. 
\end{itemize}

For case (I), we conducted an exhaustive search for Legendre pairs of order $117$ using the subgroup $\{1, 16, 22\}$ in 31 CPU hours. The search yielded $69{,}735{,}984$ candidate sequences passing the PSD test, among them 192 Legendre pairs of lengths $117$ were found. These pairs occur in 48 four-cycles (bipartite complete graphs~$K_{2,2}$): by a four-cycle we mean four sequences $A,B,C,D$ forming four Legendre pairs $(A,B)$, $(B,C)$, $(C,D)$, and $(D,A)$.
However, these 192 Legendre pairs contain some redundancy due to symmetries. Denote by $\sigma$ the cyclic (forward) shift and by $\rho$ the reverting of a sequence, and assume that $(A,B)$ is a Legendre pair. Then also $\bigl(A,\rho^i(\sigma^j(B))\bigr)$ is a Legendre pair for any choice of $i$ and~$j$, because the sequence of PAF values is invariant under shifting and reverting, i.e., $\PAF(B,s)=\PAF(\rho^i(\sigma^j(B)),s)$ for any~$s$. Note that most of these pairs will not be found during this exhaustive search, because they are not compatible with the imposed orbit structure. The only operation that is compatible is $A\mapsto\rho(\sigma(A))$, because the set of orbits is invariant under $i\mapsto\ell-i$. Thus, we can design four Legendre pairs from the four sequences
\[
  A,\quad \rho(\sigma(A)),\quad B,\quad \rho(\sigma(B)).
\]
Ten out of the $192$ Legendre pairs are given below in the form $(A_{k},B_{k}), k=1,\ldots,10$. Moreover, all 10 Legendre pairs of length $117$ shown below, have $117/3=39$-th PSD values equal to $[64,172]$. Among the remaining Legendre pairs of length $117$, some also have $117/3=39$-th PSD values equal to $[28,208]$. Algorithm~\ref{algo.pam} explains why there are no pairs with $[112,124]$, see Example~\ref{ex.pam_117_H1}.
Taking advantage of this property computationally, results in significant gains in CPU time, because we first use this property as a fast filtering mechanism (using exact arithmetic), before applying the computationally expensive full PSD test (using floating-point arithmetic). We also used the PSD constancy property over the orbits, see \cite{DK:DCC:2015}, in order to compute solely one PSD value per orbit.

In the following list, each Legendre pair $(A,B)$ is given by two index sets $I_A$ and $I_B$. The positions~$k$ where the sequence~$A$ equals~$1$, i.e., $a_k=1$, are given by $\bigcup_{i\in I_A} H_1\cdot i$, and the sequence~$A$ equals $-1$ at all other positions. Analogously, the index set $I_B$ encodes the $\{-1,+1\}$-sequence~$B$.
\begin{enumerate}
    \item 
    $I_{A_1} = \{1, 3, 4, 7, 8, 13, 14, 17, 19, 24, 28, 29, 36, 39, 40, 47, 51, 56, 63, 78, 95\}$ \\
    $I_{B_1} = \{2, 5, 7, 9, 13, 14, 18, 19, 20, 24, 34, 36, 39, 40, 42, 47, 56, 58, 73, 78, 79\}$
    \item
    $I_{A_2} = \{1, 4, 8, 10, 12, 18, 20, 29, 34, 35, 36, 39, 40, 47, 56, 57, 58, 63, 73, 78, 95\}$ \\
    $I_{B_2} = \{3, 5, 6, 7, 8, 9, 10, 12, 13, 14, 18, 19, 20, 26, 28, 39, 40, 41, 47, 56, 78\}$
    \item 
    $I_{A_3} = \{1, 2, 4, 6, 7, 8, 10, 14, 18, 29, 34, 36, 39, 47, 51, 56, 63, 73, 78, 79, 95\}$ \\
    $I_{B_3} = \{2, 3, 5, 7, 10, 12, 13, 14, 20, 24, 26, 28, 34, 36, 39, 40, 41, 47, 56, 63, 78\}$
    \item
    $I_{A_4} = \{2, 3, 6, 7, 9, 19, 21, 26, 29, 34, 39, 40, 41, 47, 56, 58, 63, 73, 78, 79, 95\}$ \\
    $I_{B_4} = \{1, 2, 3, 4, 5, 14, 17, 18, 25, 26, 29, 35, 36, 39, 40, 56, 57, 58, 63, 73, 78\}$ 
    \item
    $I_{A_5} = \{2, 3, 9, 10, 17, 18, 19, 20, 25, 34, 36, 39, 41, 47, 56, 57, 58, 73, 78, 79, 95\}$ \\
    $I_{B_5} = \{1, 2, 4, 8, 9, 13, 14, 17, 21, 26, 29, 39, 40, 42, 56, 57, 58, 63, 73, 78, 95\}$
    \item 
    $I_{A_6} = \{1, 2, 4, 5, 6, 8, 13, 17, 18, 19, 21, 34, 36, 39, 40, 41, 47, 51, 56, 73, 78\}$ \\
    $I_{B_6} = \{2, 4, 7, 8, 9, 10, 13, 18, 24, 25, 29, 35, 39, 40, 51, 56, 63, 73, 78, 79, 95\}$
    \item
    $I_{A_7} = \{2, 5, 6, 7, 8, 10, 13, 17, 18, 20, 21, 36, 39, 40, 41, 51, 58, 73, 78, 79, 95\}$ \\
    $I_{B_7} = \{3, 4, 5, 7, 10, 14, 17, 18, 26, 28, 29, 35, 36, 39, 40, 41, 57, 63, 78, 79, 95\}$
    \item
    $I_{A_8} = \{3, 4, 5, 7, 8, 18, 21, 24, 25, 28, 29, 34, 39, 40, 41, 42, 47, 56, 73, 78, 95\}$ \\
    $I_{B_8} = \{3, 9, 14, 17, 19, 21, 25, 28, 29, 34, 35, 39, 40, 47, 51, 57, 58, 73, 78, 79, 95\}$
    \item 
    $I_{A_9} = \{1, 2, 4, 6, 7, 9, 10, 12, 13, 14, 18, 28, 29, 34, 35, 39, 41, 42, 56, 78, 95\}$ \\
    $I_{B_9} = \{5, 6, 8, 9, 10, 13, 14, 19, 20, 25, 28, 34, 36, 39, 41, 51, 56, 58, 63, 73, 78\}$
    \item 
    $I_{A_{10}} = \{1, 2, 5, 7, 8, 9, 19, 20, 24, 29, 35, 36, 39, 40, 51, 58, 63, 73, 78, 79, 95\}$ \\
    $I_{B_{10}} = \{5, 7, 9, 10, 13, 14, 17, 20, 21, 26, 28, 35, 39, 40, 42, 56, 57, 63, 78, 79, 95\}$
\end{enumerate}
We also list the above 10 Legendre pairs of pairs length $117$ in a more succinct manner, using the lexicographic ranks of the subsets encoding the positions of $+1's$:
$$
\begin{array}{ll}
(10327421105, 25363140085), & (15300082821, 29082145926), \\
(5172847060, 20669267508),  & (21265971921, 810444739),   \\
(22124932714, 6023154169),  & (4370665803, 24003646556),  \\
(24634133277, 27568254144), & (27457918899, 31248697558), \\
(5218049000, 33814036464),  & (6896605532, 34222709639).  \\
\end{array}
$$
More specifically, these are the $19$-element subsets of $\{1,\ldots,38\}$, ranked lexicographically from $0$ to $\binom{38}{19}-1 = 35{,}345{,}263{,}800-1$.
See \cite{Kreher_Stinson} for ranking and unranking algorithms for $k$-element subsets and other useful combinatorial structures. For example, the integer $10327421105$ encodes the subset
\[
  \{1, 3, 4, 7, 8, 12, 13, 14, 16, 19, 22, 23, 26, 27, 30, 31, 32, 35, 38\}
  \subset \{1, \dots, 38\},
\]
which corresponds to $I_{A_1}$ (using the order of the orbits as displayed above).

For case~(II), we conducted an exhaustive search for Legendre pairs of order $117$ using the subgroup $\{1, 16, 22\}$. The search yielded $139{,}471{,}968$ candidate sequences passing the PSD test, among them $768$ Legendre pairs of lengths $117$ were found. These $768$ Legendre pairs occur in $48$ $K_{4,4}$ bipartite graphs. Similar to case (I), we can explain this phenomenon via the underlying symmetries.

Recall the notations $\sigma$ and $\rho$ for the cyclic shift and reverting of a sequence. In case~(I) we chose both $1$-orbits, and therefore all sequences $A$ in the search space satisfied $A_{39} = A_{78} = 1$ and $A_{117} = -1$. In contrast, we choose in case~(II) only one $1$-orbit and hence the sequences in the search space have $(A_{39}, A_{78}, A_{117})$ equal to $(-1,1,1)$ or $(1,-1,1)$. A sequence of the first type ($A_{39}=-1$) can be mapped to one of the second type ($A_{78}=-1$) by $\sigma^{39}$, notably without leaving the search space, because the set of orbits is invariant under $i \mapsto i + 39\; (\mod 117)$. Hence one finds the following four sequences with identical PAF values in the search space of case~(II):
\[
  A,\quad \sigma^{39}(A),\quad \rho(\sigma(A)),\quad \rho(\sigma^{40}(A)).
\]
Combining any of these four sequences with any of the four sequences with complementary PAF sequence forms a Legendre pair. This explains the occurrence of $K_{4,4}$ bipartite graphs.

Also note that for any sequence $A$ with $A_{39}=-1$ in case~(II), we find the sequence $\sigma^{-39}(A)$ in the search space of case~(I).
Ignoring the symmetries, i.e., picking one representative from each class, yields $48$ Legendre pairs which are non-equivalent under shifting and reverting.

\subsubsection{Legendre pairs of length $117$ via $H_2$}
The subgroup $H_2 = \{1, 40,79 \}$ of order 3 of $\Z_{117}^\star$ acts on $\Z_{117}$ and yields $38$ size 1 orbits and $26$ size 3 orbits, which gives a lot of possible combinations to build subsets of size $58$ (or $59$). We have not been able to construct Legendre pairs of length $117$ using $H_2$, possibly because we did not implement an exhaustive search in this case.

\subsubsection{Legendre pairs of length $117$ via $H_3$}
The subgroup $H_3 = \{1, 55, 100 \}$ of order 3 of $\Z_{117}^\star$ acts on $\Z_{117}$ and yields $8$ size 1 orbits and $36$ size 3 orbits, which gives $3$ possible combinations to build subsets of size $59$:
\begin{itemize}
    \item[(a)] $8 \cdot 1 + 17 \cdot 3 = 59$ with search space of size $8{,}597{,}496{,}600$  (2.7 CPU hours); there are $2{,}812{,}308$ sequences passing the PSD test.
    \item[(b)] $5 \cdot 1 + 18 \cdot 3 = 59$ with search space of size $508{,}207{,}576{,}800$  (95 CPU hours); there are $50{,}685{,}120$ sequences passing the PSD test.
    \item[(c)] $2 \cdot 1 + 19 \cdot 3 = 59$ with search space of size $240{,}729{,}904{,}800$  (49 CPU hours); there are $36{,}699{,}600$ sequences passing the PSD test.
\end{itemize}
In addition, there are also $3$ possible combinations to build blocks of size $58$:
\begin{itemize}
    \item[(d)] $7 \cdot 1 + 17 \cdot 3 = 58$ with search space of size $68{,}779{,}972{,}800$  (13.5 CPU hours); there are $10{,}485{,}600$ sequences passing the PSD test.
    \item[(e)] $4 \cdot 1 + 18 \cdot 3 = 58$ with search space of size $635{,}259{,}471{,}000$  (119 CPU hours); there are $63{,}356{,}400$ sequences passing the PSD test.
    \item[(f)] $1 \cdot 1 + 19 \cdot 3 = 58$ with search space of size $68{,}779{,}972{,}800$  (22 CPU hours); there are $22{,}498{,}464$ sequences passing the PSD test.
\end{itemize}
No Legendre pair of length $117$ whose multiplier group
contains $H_3$ exists.

\subsubsection{Legendre pairs of length $117$ via $H_4$}

The subgroup $H_4 = \{1, 61, 94\}$ of order 3 of $\Z_{117}^\star$ acts on $\Z_{117}$ and yields a search space of size: $\binom{2}{2} \cdot \binom{38}{19} = 35{,}345{,}263{,}800$, because there are $38$ orbits of size $3$ and $2$ orbits of size~$1$, and we need $19$ orbits of size $3$ and $2$ orbits of size~$1$ to make a subset of size $19 \cdot 3 + 2 \cdot 1 = 59 = (117+1)/{2}$.
We found $240$ Legendre pairs of length $117$ via an exhaustive search. Some of them are shown below in LexRank form:
$$
\begin{array}{@{}lll@{}}
(8221110983, 12044164377), & 
(12702071296, 15372978390), & 
(23944768832, 15178414396), \\ 
(20338660993, 90051589), &
(7146518669, 23738703053), &
(3073133857, 30770050335), \\
(32540516078, 3097218289), &
(33749219312, 4797783684), &
(5422010999, 7269176966).
\end{array}
$$
Among the $240$ pairs, there are $144$ pairs with $[\PSD(A,39),\PSD(B,39)]$ equal to  $[64,172]$ and $96$ pairs with $[\PSD(A,39),\PSD(B,39)]$ equal to $[28,208]$.

\noindent
A search with block size $58$ delivered $960$ Legendre pairs of length~$117$. Analogous to subgroup $H_1$, any of these pairs can be obtained from others by shifting and reverting one or both sequences. Hence, we can extract $60$ Legendre pairs using~$H_4$ which are non-equivalent under shifting and reverting.

\subsection{Legendre pairs of length $129$}
\label{sec:129}

We executed Algorithm~\ref{algo.pam_pbm_spectrum} for Legendre pairs of length $\ell = 129 = 3 \cdot 43$ and obtained that the spectrum of possible pairs of values for $\PSD(A,43)$ and $\PSD(B,43)$ is made up of only 5 pairs:
$$
  \bigl[\PSD(A,43), \PSD(B,43)\bigr] \in
  \bigl\{ [4, 256], [16, 244], [52, 208], [64, 196], [112, 148] \bigr\} .
$$ 
There is one  subgroup of order 3 in $\Z_{129}^\star$
$$
   H = \{ 1, 49, 79 \} 
$$
acting on $\Z_{129}$ and yielding a search space of size
$\binom{2}{2} \cdot \binom{42}{21} = 538{,}257{,}874{,}440$. Since there  
are $42$ orbits of size $3$ and $2$ orbits of size $1$, we need 
$21$ orbits of size $3$ and $2$ orbits of size $1$, to make a subset of 
size $21 \cdot 3 + 2 \cdot 1 = 65 = (129+1)/{2}$.
The $42+2$ orbits of the action of $H = \{1, 49, 79\}$ 
on $\Z_{129}$ are
$$
\begin{array}{l@{\qquad}l@{\qquad}l}
H \cdot 1 = \{1, 49, 79\}, & H \cdot 2 = \{2, 29, 98\}, & H \cdot 3 = \{3, 18, 108\},\\
H \cdot 4 = \{4, 58, 67\}, & H \cdot 5 = \{5, 8, 116\}, & H \cdot 6 = \{6, 36, 87\},\\
H \cdot 7 = \{7, 37, 85\}, & H \cdot 9 = \{9, 54, 66\}, & H \cdot 10 = \{10, 16, 103\},\\
H \cdot 11 = \{11, 23, 95\}, & H \cdot 12 = \{12, 45, 72\}, & H \cdot 13 = \{13, 121, 124\},\\ 
H \cdot 14 = \{14, 41, 74\}, & H \cdot 15 = \{15, 24, 90\}, & H \cdot 17 = \{17, 53, 59\},\\
H \cdot 19 = \{19, 28, 82\}, & H \cdot 20 = \{20, 32, 77\}, & H \cdot 21 = \{21, 111, 126\},\\
H \cdot 22 = \{22, 46, 61\}, & H \cdot 25 = \{25, 40, 64\}, & H \cdot 26 = \{26, 113, 119\},\\
H \cdot 27 = \{27, 33, 69\}, & H \cdot 30 = \{30, 48, 51\}, & H \cdot 31 = \{31, 100, 127\},\\
H \cdot 34 = \{34, 106, 118\}, & H \cdot 35 = \{35, 38, 56\}, & H \cdot 39 = \{39, 105, 114\},\\
H \cdot 42 = \{42, 93, 123\}, & H \cdot 44 = \{44, 92, 122\}, & H \cdot 47 = \{47, 101, 110\},\\
H \cdot 50 = \{50, 80, 128\}, & H \cdot 52 = \{52, 97, 109\}, & H \cdot 55 = \{55, 88, 115\},\\
H \cdot 57 = \{57, 84, 117\}, & H \cdot 60 = \{60, 96, 102\}, & H \cdot 62 = \{62, 71, 125\},\\
H \cdot 63 = \{63, 75, 120\}, & H \cdot 65 = \{65, 89, 104\}, & H \cdot 68 = \{68, 83, 107\},\\
H \cdot 70 = \{70, 76, 112\}, & H \cdot 73 = \{73, 91, 94\}, & H \cdot 78 = \{78, 81, 99\},\\
H \cdot 43 = \{43\}, &  H \cdot 86 = \{86\} . &
\end{array}
$$

We conducted an exhaustive search for Legendre pairs of order $129$
using the subgroup $H = \{ 1, 49, 79 \}$ in 431 CPU hours. The search was done in parallel on 16 threads and yielded output files of total size 80 gigabytes, containing more than $460$ million sequences which passed the PSD test; among them 112 Legendre pairs of length $129$ were found. Analogous to $\ell=117$, this list can be condensed to $28$ pairs, where the remaining ones are obtained by symmetry.

Here are two Legendre pairs of length $129$, given by index sets 
$I_A, I_B$:
\begin{enumerate}
    \item $I_A = \{1, 2, 5, 13, 17, 19, 21, 22, 25, 26, 27, 34, 39, 43, 50, 55, 60, 62, 63, 68, 73, 78, 86\}$ \\
    $I_B = \{1, 3, 11, 12, 13, 17, 21, 26, 31, 34, 35, 42, 43, 47, 50, 52, 57, 60, 62, 68, 70, 78, 86\}$
    \item $I_A = \{1, 2, 5, 13, 17, 19, 21, 22, 25, 26, 27, 34, 39, 43, 50, 55, 60, 62, 63, 68, 73, 78, 86\}$ \\
    $I_B = \{1, 2, 3, 4, 5, 6, 10, 11, 12, 17, 19, 20, 21, 22, 27, 30, 34, 43, 50, 57, 70, 73, 86\}$
\end{enumerate}
Both the above Legendre pairs of length $129$ have $\bigl[\pa{43},\pb{43}\bigr] = [148,112]$. In fact, all the 112 Legendre pairs of length $129$ that we found have this property. This is partly explained by applying Algorithm~\ref{algo.pam}, where
\[
  4,\, 76,\, 112,\, 148,\, 256,\, 292,\, 364,\, 400, \dots
\]
are returned as potential values for $\pa{43}$ and $\pb{43}$, leaving only two possible pairs, namely $[4, 256]$ and $[112, 148]$.


\subsection{Legendre pairs of length $147$}
\label{sec:147}

We executed Algorithm~\ref{algo.pam_pbm_spectrum} for Legendre pairs of length $\ell = 147 = 3 \cdot 49$ and obtained that the spectrum of possible pairs of values for $\PSD(A,49)$ and $\PSD(B,49)$ is made up of only 6 pairs
$$
  \bigl[\PSD(A,49), \PSD(B,49)\bigr] \in \bigl\{ [4, 292], [28, 268], [52, 244], [100, 196], [124, 172], [148, 148] \bigr\} .
$$ 
Algorithm~\ref{algo.pam} further shows that in fact only two of the above six pairs (namely the first and the last one) are compatible with the particular orbit structure induced by the subgroup $H$ of order~$3$ in $\Z_{147}^\star$
$$
   H = \{ 1, 67, 79 \} 
$$
acting on $\Z_{147}$. This information prevents us from conducting redundant computations and helps us target the search more narrowly.
The $50$ orbits of the action of $H = \{1, 67, 79\}$ 
on $\Z_{147}$ are
$$
\begin{array}{l@{\qquad}l@{\qquad}l}
H \cdot 1 = \{1, 67, 79\}, & H \cdot 2 = \{2, 11, 134\}, & H \cdot 3 = \{3, 54, 90\} , \\
H \cdot 4 = \{4, 22, 121\}, & H \cdot 5 = \{5, 41, 101\}, & H \cdot 6 = \{6, 33, 108\}, \\
H \cdot 7 = \{7, 28, 112\}, & H \cdot 8 = \{8, 44, 95\}, & H \cdot 9 = \{9, 15, 123\}, \\
H \cdot 10 = \{10, 55, 82\}, & H \cdot 12 = \{12, 66, 69\}, & H \cdot 13 = \{13, 136, 145\}, \\
H \cdot 14 =\{14, 56, 77\}, & H \cdot 16 =\{16, 43, 88\}, & H \cdot 17 =\{17, 20, 110\}, \\
H \cdot 18 =\{18, 30, 99\}, & H \cdot 19 =\{19, 31, 97\}, & H \cdot 21 =\{21, 42, 84\}, \\
H \cdot 23 =\{23, 53, 71\}, & H \cdot 24 =\{24, 132, 138\}, & H \cdot 25 =\{25, 58, 64\}, \\
H \cdot 26 =\{26, 125, 143\}, & H \cdot 27 =\{27, 45, 75\}, & H \cdot 29 =\{29, 32, 86\}, \\
H \cdot 34 =\{34, 40, 73\}, & H \cdot 35 =\{35, 119, 140\}, & H \cdot 36 =\{36, 51, 60\}, \\
H \cdot 37 =\{37, 127, 130\}, & H \cdot 38 =\{38, 47, 62\}, & H \cdot 39 =\{39, 114, 141\}, \\
H \cdot 46 =\{46, 106, 142\}, & H \cdot 48 =\{48, 117, 129\}, & H \cdot 49 =\{49\}, \\
H \cdot 50 =\{50, 116, 128\}, & H \cdot 52 =\{52, 103, 139\}, & H \cdot 57 =\{57, 93, 144\}, \\
H \cdot 59 =\{59, 104, 131\}, & H \cdot 61 =\{61, 115, 118\}, & H \cdot 63 =\{63, 105, 126\}, \\
H \cdot 65 =\{65, 92, 137\}, & H \cdot 68 =\{68, 80, 146\}, & H \cdot 70 =\{70, 91, 133\} \\
H \cdot 72 =\{72, 102, 120\}, & H \cdot 74 =\{74, 107, 113\}, & H \cdot 76 =\{76, 94, 124\}, \\
H \cdot 78 =\{78, 81, 135\}, & H \cdot 83 =\{83, 89, 122\}, & H \cdot 85 =\{85, 100, 109\}, \\
H \cdot 87 =\{87, 96, 111\}, & H \cdot 98 =\{98\}. & \\
\end{array}
$$

Here is a Legendre pair $(A,B)$ of length $\ell = 147$, given by two index sets $I_A$ and $I_B$. The positions~$k$ where the sequence~$A$ equals~$1$, i.e., $a_k=1$, are given by $\bigcup_{i\in I_A} H \cdot i$, and the sequence~$A$ equals $-1$ at all other positions. Analogously, the index set $I_B$ encodes the $\{-1,+1\}$-sequence~$B$:
\begin{align*}
    I_A &= \{ 1, 2, 3, 5, 7, 8, 10, 14, 16, 17, 19, 21, 27, 35, 38, 39, 49, 52, 57, 61, 70, 72, 74, 83, 87, 98  \}, \\
    I_B &= \{ 1, 2, 6, 7, 9, 10, 12, 16, 17, 19, 23, 24, 26, 35, 39, 46, 48, 49, 50, 59, 65, 68, 70, 78, 85, 98 \}.
\end{align*}
The LexRank encoding of the Legendre pair $(A,B)$ of length $\ell = 147$ is
$$
(2279447240326, 6981583007090).
$$
This Legendre pair $(A,B)$ of length $\ell = 147$ has $\bigl[\pa{49}$, $\pb{49}\bigr]=[148,148]$, the second pair of values predicted by Algorithms~\ref{algo.pam_pbm_spectrum} and~\ref{algo.pam}.

We also give three Legendre pairs of length $\ell = 147$ with $\bigl[\pa{49}$, $\pb{49}\bigr]=[4,292]$, the first pair of values predicted by Algorithms~\ref{algo.pam_pbm_spectrum} and~\ref{algo.pam}:

    \begin{enumerate}
        \item $(1685512212865, 3612702197526)$,
        \item $(2926263388957, 265692014998)$,
        \item $(4357037511235, 3728601853735)$.
    \end{enumerate}

We remark that the combination of Algorithms~\ref{algo.pam_pbm_spectrum} and~\ref{algo.pam} was of critical importance, in order to traverse the first portion (15\%) of the huge search space of 32 trillion elements and find the Legendre pairs of length $\ell = 147$. 

\subsection{Legendre pairs of length $133$}
\label{sec:133}

We used the subgroup of order $3$, $H = \{1,11,121\}$, acting on $\Z_{133}$. This yields a search space of size
$\binom{44}{22}=2{,}104{,}098{,}963{,}720$ elements. The computation was stopped after $20\%$ of the search space was traversed, in 707 hours of CPU time.
The output files grew to a total size of 108 gigabytes and $5$ new Legendre pairs of length $133$ were discovered (we display their lexicographic rank for a $22$-subset of the $44$ orbits of size~$3$, but this time these indices give the positions of the $-1$'s):
\begin{enumerate}
    \item $(128572618842, 210086022915)$,
    \item $(17644506807, 41167368128)$,
    \item $(179364459458, 27235734754)$, 
    \item $(213277890206, 251235525902)$,
    \item $(272147218211, 279717372516)$.
\end{enumerate}

These five Legendre pairs can be used to make 
Hadamard matrices of order $2 \cdot 133 + 2 = 268$, via the two circulant core template array in \cite{FGS:2001}. The order $268$ was the smallest open order for Hadamard matrices until 1985~\cite{Sawade:HM268:1985}. 

These five Legendre pairs have integer PSD values at integer multiples of the prime factor $19$ of $\ell = 133$. More specifically, using the above numbering we have
\begin{enumerate}
    \item $\bigl[\pa{19},\pb{19}\bigr] = \bigl[\pa{38},\pb{38}\bigr] = \bigl[\pa{57},\pb{57}\bigr] = [176,92]$,
    \item $\bigl[\pa{19},\pb{19}\bigr] = \bigl[\pa{38},\pb{38}\bigr] = \bigl[\pa{57},\pb{57}\bigr] = [92,176]$,
    \item $\bigl[\pa{19},\pb{19}\bigr] = \bigl[\pa{38},\pb{38}\bigr] = \bigl[\pa{57},\pb{57}\bigr] = [36,232]$,
    \item $\bigl[\pa{19},\pb{19}\bigr] = \bigl[\pa{38},\pb{38}\bigr] = \bigl[\pa{57},\pb{57}\bigr] = [92,176]$,
    \item $\bigl[\pa{19},\pb{19}\bigr] = \bigl[\pa{38},\pb{38}\bigr] = \bigl[\pa{57},\pb{57}\bigr] = [92,176]$.
\end{enumerate}

Using Proposition~\ref{Proposition:integerPSD} 
with $\ell = 133, m = 19, n = 7$, we are able to ascertain the cause of this property.

\begin{itemize}
  \item For the 3rd Legendre pair of length $133$ we found, using the notations of Proposition~\ref{Proposition:integerPSD} we have:
  \begin{alignat*}{2}
    \mathcal{A} &= [1, 1, 1, 1, 1, 1, -5],\quad &
    \PAF(\mathcal{A},1) = \PAF(\mathcal{A},2) = \PAF(\mathcal{A},3) & = -5, \\
    \mathcal{B} &= [-1, -1, 5, -1, 5, 5, -11],\quad &
    \PAF(\mathcal{B},1) = \PAF(\mathcal{B},2) = \PAF(\mathcal{B},3) & = -33. 
    \end{alignat*}
    Therefore, by applying Proposition~\ref{Proposition:integerPSD}, we obtain (note that $-5-33 = 2 \cdot (-19) $)
  \begin{align*}
  \PSD(A,19) = \PSD(A,38) = \PSD(A,57) & = p_2(\mathcal{A}) - \PAF(\mathcal{A},1) = 31 + 5 = 36, \\ 
  \PSD(B,19) = \PSD(B,38) = \PSD(B,57) & = p_2(\mathcal{B}) - \PAF(\mathcal{B},1) = 199 + 33 = 232.
  \end{align*}

  \item For the 5th Legendre pair of length $133$ we found, using the notations of Proposition \ref{Proposition:integerPSD} we have:
  \begin{alignat*}{2}
  \mathcal{A} &= [1, 1, -3, 1, -3, -3, 7],\quad &
  \PAF(\mathcal{A},1) = \PAF(\mathcal{A},2) = \PAF(\mathcal{A},3) &= -13, \\
  \mathcal{B} &= [-5, -5, 5, -5, 5, 5, 1],\quad &
  \PAF(\mathcal{B},1) = \PAF(\mathcal{B},2) = \PAF(\mathcal{B},3) &= -25.
  \end{alignat*}
  Therefore, by applying Proposition~\ref{Proposition:integerPSD}, we obtain (note that $-13-25 = 2 \cdot (-19) $)
  \begin{align*}
  \PSD(A,19) = \PSD(A,38) = \PSD(A,57) & = p_2(\mathcal{A}) - \PAF(\mathcal{A},1) = 79 + 13 = 92,\\
  \PSD(B,19) = \PSD(B,38) = \PSD(B,57) & = p_2(\mathcal{B}) - \PAF(\mathcal{B},1) = 151 + 25 = 176.
  \end{align*}

\end{itemize}

\section{Conclusion}
We prove a proposition that connects constant-PAF sequences and the corresponding Legendre pairs with integer PSD values.
We update the list of open lengths for Legendre pairs, in \cite{BD:2020}. In particular, we furnish the first ever examples of Legendre pairs of the four open lengths  $117, 129, 133, 147$. In the case of the three open lengths $117, 129, 147$, we make extensive use of two new algorithms. Our first algorithm yields the determination of the complete spectrum of the (resp. $39$-th, $43$-rd, $49$-th) value of the discrete Fourier transform for Legendre pairs. In fact, our algorithm yields the complete spectrum of the $({\ell}/{3})$-rd value of the discrete Fourier transform for Legendre pairs of lengths $\ell \equiv 0 \,\, (\mod 3)$. Our second algorithm exploits the particular orbit structure induced by specific subgroups of their multiplier group, to disqualify certain elements of the spectrum determined by the first algorithm. The combination of both algorithms for Legendre pairs of lengths $\ell \equiv 0 \, (\mod 3)$ was a decisive factor in discovering Legendre pairs of the three open lengths $117, 129, 147$.        
A Legendre pair of length $\ell = 77$ was reported in 2020 in \cite{JT:2020}, see \cite{JT:2021}. Therefore, the state-of-the-art list of twelve integers in the range $ < 200$ for which the question of existence of Legendre pairs is still unsolved is
$$
85, 87, 115, 145, 159, 161, 169, 175, 177, 185, 187, 195.
$$
For $\ell = 87$, the order-$7$ subgroup $\{1, 7, 16, 25, 49, 52, 82\}$, and both subgroups of order~$4$ (namely, $\{1, 17, 28, 41\}$ and $\{1, 28, 46, 70\}$) did not yield any solutions by exhaustive search. We also tried some of the subgroups of order~$2$ (there are three of them: $\{1,28\},\{1,59\},\{1,86\}$), which admit many possible combinations of their associated subsets, some of whose search spaces being beyond our computational resources, but without success. Therefore, even twenty years after the fundamental paper \cite{FGS:2001} for Legendre pairs appeared, there are still interesting questions and open problems to ponder in this area.

\paragraph{Acknowledgment} We are grateful for the detailed and pertinent referee comments that contributed to improving our paper. In particular, one of the referees provided a generalized version of our original Proposition~\ref{Proposition:integerPSD} and a significantly simplified version of its proof.


\bibliographystyle{plain}
\bibliography{Legendre_pairs}

\end{document}